\documentclass[12pt]{amsart}
\usepackage[top=2in, bottom=1.5in, left=1in, right=1in]{geometry}                
\usepackage{graphicx}
\usepackage{amssymb}
\usepackage{epstopdf}
\usepackage[latin1]{inputenc}
\usepackage{subfigure}
\usepackage{amsfonts}
\usepackage{amsmath}

\newcommand{\mc}[1]{\mathcal{#1}}
\newcommand{\mf}[1]{\mathfrak{#1}}

\usepackage{amssymb}

\newcommand{\R}{\mathbb{R}}

\newcommand{\Z}{\mathbb{Z}}

\newcommand{\ga}{\alpha}

\newcommand{\gb}{\beta}

\newcommand{\ord} [1]{| #1 |}
\newcommand{\h}{\mf h}
\newcommand{\eb}[1]{\textbf{\emph{{#1}}}}
\newcommand{\ed}[1]{\stackrel{#1}{\longrightarrow}}
\newcommand{\HG}{H^{*}_{T}(\h_\gamma)}

\newtheorem{theorem}{Theorem}[section]
\newtheorem{corollary}[theorem]{Corollary}
\newtheorem{lemma}[theorem]{Lemma}
\newtheorem{proposition}[theorem]{Proposition}

\newtheorem{definition}[theorem]{Definition}

\newtheorem{example}[theorem]{Example}

\title{Divided Difference Operator for  the Highest root Hessenberg variety}
\author{Nicholas Teff}
\address{Department of Mathematics, University of Iowa, 14 MacLean Hall, Iowa City, IA 52242-1419}
\email{nicholas-teff@uiowa.edu}
\thanks{Partially supported by the University of Iowa Department of  Mathematics NSF VIGRE grant DMS-0602242}
\begin{document}
\maketitle
\begin{abstract}
We construct a divided difference operator using GKM theory.  This generalizes the classical divided difference operator for the cohomology of the complete flag variety.  This construction proves a special case of a recent conjecture of  Shareshian and Wachs.  Our methods are entirely combinatorial and algebraic, and rely heavily on the combinatorics of root systems and Bruhat order.
\end{abstract}

\section{Intoduction}

This article is an extended abstract of the article \cite{Te-HRDD} of the same title.  Most of the details of the proofs are omitted.

A classical problem of Schubert calculus is to define explicit classes $\mc S^{[w]}$ to represent Schubert varieties in cohomology rings of a partial flag variety.  For geometric reasons these classes form an additive basis for the cohomology.    In equivariant cohomology this problem reduces to finding the polynomials $\mc S^{[w]}([v])$ which are nonzero only if $[v] \ge [w]$ in Bruhat order.  For more general spaces the uniqueness or even existence of generalized Schubert classes named \emph{flow-up classes} is not guaranteed.  When they exist it is natural to ask for some combinatorial formula defining the polynomials.  This is the type of question we adress here.

A motivating example for our work is the complete flag variety $G/B$.  By a combinatorial construction called \emph{GKM theory} (named after Goresky, Kottwitz and MacPherson) the equivariant cohomology is computed directly from the \emph{Bruhat graph}  $\Gamma_{W}$ of the Weyl group $W$ (for definitions see Section \ref{hgraph}) \cite{GKM, Ty-RS}. The Schubert classes classes are constructed by divided difference operators 
\[\partial_{i}: \mc S^{w}(u) \longmapsto \frac{\mc S^{w}(u) - s_{i}\mc S^{w}(s_{i}u)}{\ga_{i}}.\]   

These operators were first introduced by Bernstein, Gelfand, and Gelfand; and Demazure for ordinary cohomology, and Kostant and Kumar generalized them to equivariant cohomology \cite{BGG-DGP,De-DO,KoKu-HR}.  More recently, employing GKM theory Tymoczko uses a left action of $W$ and defines new divided difference operators \cite{Ty-RS}.  Flow-up classes for $G/B$ are unique, so this construction agrees with the earlier work.   

A benefit of divided difference operators is that they are recursive maps.  This means if $\mc S^w$ is known and $s_iw < w$, then $\mc S^{s_iw} := \partial_{i}\mc S^{w}$.  Billey uses this recursion of the Kostant and Kumar operators to define a closed combinatorial formula for the polynomial $\mc S^{w}(v)$ \cite{Bi-BF}.  Billey's formula is a positive formula involving the reduced expressions of $w$ obtained as a subexpression of a fixed reduced expression for $v$ \cite[Theorem 3]{Bi-BF}.  

In this paper \emph{GKM rings} (a combinatorial analog of equivariant cohomology) are defined for certain subgraphs of the Bruhat graph.  As with the Bruhat graph these rings construct the equivariant cohomology of algebraic varieties called the \emph{regular semisimple Hessenberg varieties}.  Two important examples of regular semisimple Hessenberg varieties are the complete flag variety $G/B$ and the toric variety associated to the Coxeter complex \cite{dMPS-HV}.

Hessenberg varieties were first arose in numerical analysis in the context of calculating the Hessenberg form of a matrix, and have received recent attention in the work of Tymoczko generalizing Springer theory to nilpotent Hessenberg varieties \cite{Sp-SRGF, Ty-PH}.  The cohomology ring of regular semisimple Hessenberg varieties carry a representation of $W$, of which little is known.  In fact, it remains an open question when $W \cong \mf S_{n}$ the symmetric group.  In this case your author has provided an irreducible decomposition of this representation for a large family called \emph{parabolic Hessenberg varieties} \cite{Te-HRY}.  

In another direction, the representation for $\mf S_{n}$ has appeared in a recent conjecture of Shareshian and Wachs in their work on \emph{chromatic quasisymmetric functions} \cite[Conjecture 5.3]{ShWa-CQH}.  They conjecture that the under the Frobenius isomorphism between the representation ring of $\mf S_{n}$ and the ring of symmetric functions that the image of the ordinary cohomology ring is the chromatic symmetric function they study.

Our main result (Theorem \ref{BIG}) generalizes the divided difference operator for $G/B$ to what we call the \emph{highest root Hessenberg variety}.  This result is a model first step toward defining bases which would allow us to investigate the representation on the cohomology (ordinary and equivariant).  With this basis in hand we end this paper by announcing that for the highest root Hessenberg variety the Shareshian and Wachs conjecture is true (Theorem \ref{confirm}).  

Our problem originates in algebraic geometry, but our methods are combinatorial and algebraic, a primary advantage of GKM theory.  We will see the construction of divided difference operators and the flow-up classes relies heavily on Bruhat order and root systems.  In this abstract to emphasize the combinatorial nature of this construction we have left out the formal definitions of Hessenberg varieties and GKM theory.  The curious reader is directed to \cite{dMPS-HV, Te-HRY, Ty-PH} for Hessenberg varieties and to \cite{GKM, GoTo-GS, GuZa-E, Ty-RS} for GKM theory.

\section{Hessenberg graphs}
\label{hgraph}

We begin with the definition of a \emph{Weyl group $W$} \cite{Hu-RG}.  Let $V$ be a $k$-dimensional real vector space with a symmetric positive definite bilinear form $(\;,\;)$.  A \emph{reflection} in $V$ is a linear map which negates a non-zero vector $\ga \in V$ and fixes point-wise the hyperplane orthogonal to $\ga$.  A formula for the reflection through $\ga$ is $s_{\ga}(v) = v - {2(\ga,v)}{(\ga,\ga)}^{-1}\ga$.
\pagebreak

A \eb{(crystalographic) root system} in $V$ is a finite set of vectors $\Phi$ (called \emph{roots}) which satisfy the following axioms
\begin{itemize}
\item[]{(1.)} $\R\ga \cap \Phi = \pm \ga$ for all $\ga \in \Phi$;
\item[]{(2.)} $s_{\ga}\Phi = \Phi$ for all $\ga \in \Phi$;
\item[]{(3.)} $\frac{2(\ga,\gb)}{(\ga,\ga)} \in \Z$ for all $\ga, \gb \in \Phi$. 
\end{itemize}

The integer $c_{\ga \gb} := \frac{2(\ga,\gb)}{(\ga,\ga)}$ is called a \emph{Cartan integer}.  A \emph{base} $\Delta \subset \Phi$ is a basis of $V$ such that for each $\ga \in \Phi$ the coefficients of the expansion $\ga = \sum_{\Delta}c_{i}\ga_{i}$ are either all non-negative or all non-positive.

With a fixed base $\Delta$ the \emph{positve roots} $\Phi^+$ are those with all non-negative coefficients and respectively call $\Phi^{-} = - \Phi^{+}$ the negative roots.  There is a partial order ($\prec$) on $\Phi$ where $\ga \prec \gb$ means $\gb - \ga$ is a sum of positive roots.  We say $\mc I \subset \Phi$ is an \eb{ideal} if whenever $\gb \in \mc I$ and $\gb \in \Phi$ with $\gb \prec \ga$, then $\ga \in \mc I$. 

The Weyl group $W$ is the group generated by the \emph{simple reflections} $s_{i} := s_{\ga_{i}}$ for $\ga_{i} \in \Delta$.  For $w \in W$ the \emph{length $\ell(w)$} is the length of a reduced expression $w = s_{i_{1}}s_{i_{2}}\cdots s_{i_{j}}$.  Finally, the \eb{Bruhat graph $\Gamma_{W}$} has vertices $W$ and edges $u \ed{} w$ if $w = s_{\ga}u$ for $\ga \in \Phi^+$ and $w^{-1}\ga \in \Phi^-$ (or equivalently $\ell(w) > \ell(u)$), and the \eb{Bruhat order $<$} is the transitive closure of the edge relations.

\begin{example}  [The type $A_{n}$ root system.]  Consider $\R^{n+1}$ with dot product defined on the standard coordinate basis $\mathtt{t_{i}}$ for $i = 1, \cdots, n+1$.  Let $V$ be the span of the roots $\Phi = \{\mathtt{t_{i}} - \mathtt{t_{j}} : i \ne j\}$.  The simple roots are the $\mathtt{t_{i}} - \mathtt{t_{i+1}}$ and the positive roots are the $\mathtt{t_{i}} - \mathtt{t_{j}}$ for $i < j$.  The reflection in $\mathtt{t_{i}} - \mathtt{t_{j}}$, denoted $s_{(ij)}$, interchanges $\mathtt{t_{i}}$ and $\mathtt{t_{j}}$ and fixes the other $\mathtt{t_{k}}$.  Hence, mapping this reflection to the transposition $(ij)$ defines an isomporism of the Weyl group with $\mf{S}_{n+1}$.  
\end{example} 

\begin{center}
\begin{figure}[ht]
\includegraphics[width = 145mm]{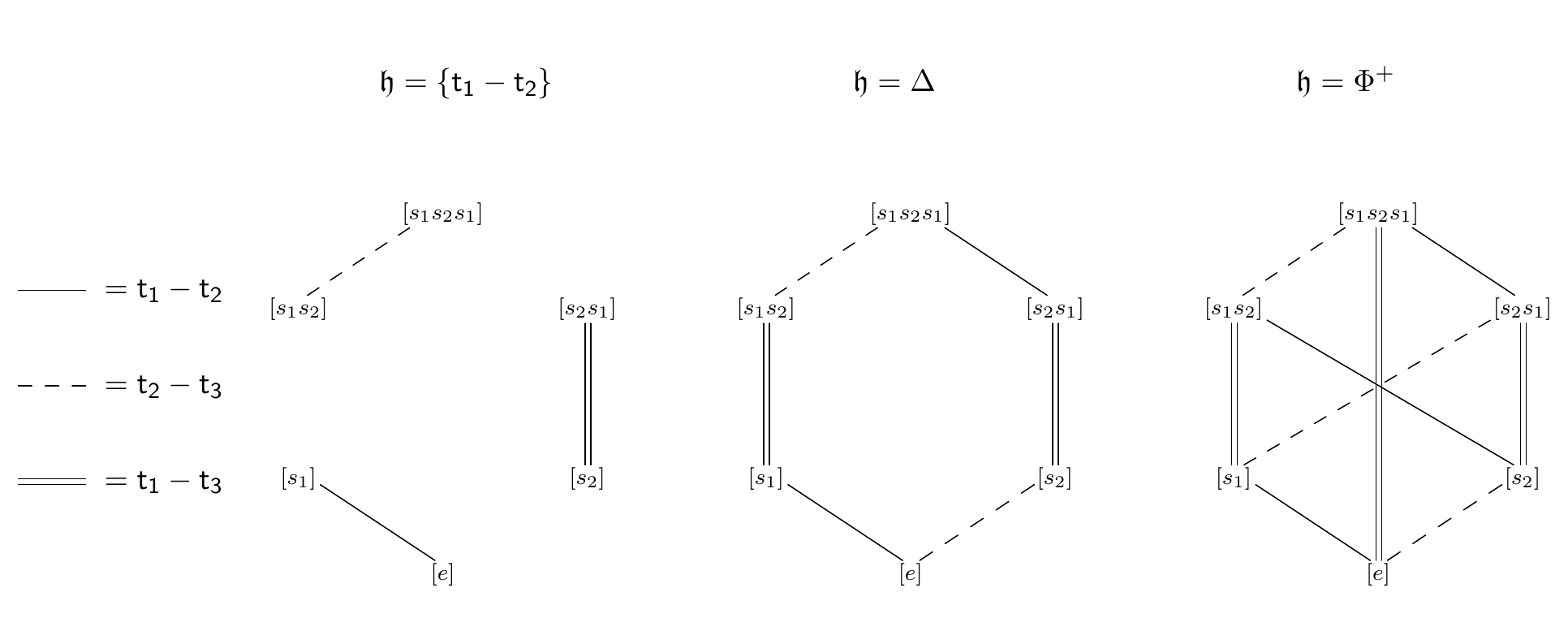}
\caption{Hessenberg graphs in type $A_{2}$}\label{mgraphs}
\end{figure}
\end{center}

\begin{definition}\label{hdef}  Let $(V, \Phi,\Delta, W)$ be as defined above.  A \eb{Hessenberg set $\h$} is the complement of an ideal of $\mc I_\h \subset \Phi^{+}$.  The \eb{Hessenberg graph $\Gamma_\h$} has vertices $W$ and edges $u \ed{} w$ if $w = s_\ga u$ for $\ga \in \Phi^+$ and $w^{-1}\ga \in -\h$.  The \eb{GKM ring of $\h$} is the subring of $\textrm{Maps}(W, \; \R[\ga_{i}, \cdots, \ga_{k}])$ defined from $\Gamma_{\h}$ by

\[ H_T^{*}(\h) = \left \{ \mc P: W \longrightarrow \R[\ga_1,\cdots, \ga_k] : \begin{array}{c} \textup{for each edge} \; w \ed{} s_{\ga}w \\ \mc P(w) - \mc P(s_{\ga}w) \in \langle \ga \rangle \end{array}\right\}.\] 
\end{definition}

The relations $\mc P(w) - \mc P(s_{\ga}w) \in \langle \ga \rangle$ are the \emph{GKM conditions}.  The GKM ring is a graded ring; we say $\mc P \in H^{k}_{T}(\h)$ if each non-zero polynomial $\mc P(w)$ is homogeneous of degree $k$.  Elements of $H_T^{*}(\h)$ are represented by labeling the vertices of $\Gamma_{\h}$ by polynomials (cf Figure \ref{dotex}). 

The GKM rings carry an action of $W$ obtained by first extending the action of $W$ on $\Phi$ to the polynomial ring $\R[\Delta] := \R[\ga_{1},\cdots,\ga_{k]}$ from which we obtain an action on $\textrm{Maps}(W, \; \R[\Delta])$ by the rule
\begin{equation}\label{dot}
(w \cdot \mc P)(u) = w\mc P(w^{-1}u)
\end{equation}
where on the right $w$ is the acting on the \emph{polynomial} $\mc P(w^{-1}u) \in \R[\Delta]$.
 
\begin{proposition}  The GKM ring $H^{*}_{T}(\h)$ is $W$-stable with respect to the action defined in Equation \ref{dot}.
\end{proposition}

\begin{proof}  Let $\mc P \in H^{*}_{T}(\h)$ and $w \in W$.  We must check the GKM conditions, \textit{i.e.} for every edge $u \ed{} s_{\ga}u$ is $(w \cdot \mc P)({u}) - (w \cdot \mc P)({s_{\ga}u}) \in \left \langle \ga \right \rangle$.  The undirected edge $u \stackrel{}{\longleftrightarrow} s_{\ga}u$ is in $\Gamma_{\h}$ if and only if the undirected edge $w^{-1}u \stackrel{}{\longleftrightarrow} s_{w^{-1}\ga}w^{-1}u \,(= w^{-1}s_{\ga}u)$ is too.  The GKM conditions ignore the edge orientation, so $\mc P({w^{-1}u}) - \mc P({w^{-1}}s_{\ga}u) \in \left \langle w^{-1}\ga\right \rangle$ is equivalent to $w\mc P({w^{-1}u}) - w\mc P({w^{-1}}s_{\ga}u) \in \left \langle \ga \right \rangle$.  The last expression is  $(w \cdot \mc P)({u}) - (w \cdot \mc P)({v})$ proving the claim.
\end{proof}

This action is easily describe on the graph when $w = s_{\ga}$ a reflection; the action of $s_{\ga}$ interchanges polynomials across edges corresponding to $s_{\ga}$ (some may have been deleted) and permutes the roots.

 \begin{center}
\begin{figure}[ht]
\includegraphics[width = 145mm]{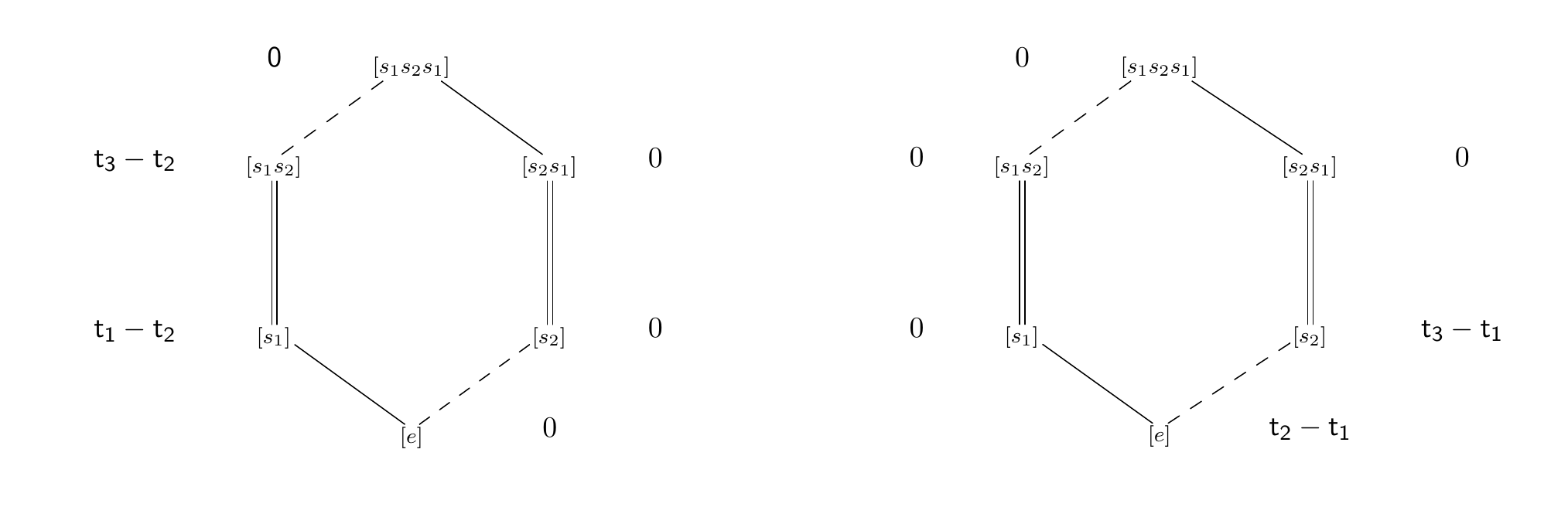}
\caption{A class and its image under $s_{1}\cdot$}\label{dotex}
\end{figure}
\end{center}

In order to study this representation we need to construct a basis of $H_{T}^*(\h)$. For the GKM ring $H_{T}^{*}(\Phi^{+})$, this basis consists of Schubert classes $\mc S^{w}$ \cite{Ty-RS}.  These are homogenous classes of degree $\ell(w)$ and the polynomial $\mc S^{w}(v)$ is nonzero only if $v > w$ in Bruhat order, \textit{i.e.} there is exists a path $w \ed{} \cdots \ed{} v$ in the Bruhat graph. 

These notions are generalized as follows.  Fix $\h$ a Hessenberg set.  The \eb{flow-up of $x \in W$} are all the vertices $y$ such that there is a path $x \ed{} \cdots \ed{} y$ in $\Gamma_{\h}$.  If $y$ is in the flow up we denote this by $x <_{\h} y$, and $\ell_{\h}(x) = k$ if there are $k$ edges ending at $x$.

\begin{definition}\label{flow}  $\mc P^{x} \in H_T^{\ell_{\h}(x)}(\h)$ is a \textbf{flow-up class} at $x \in W$ if  
\begin{itemize}
\item[(1)] $\displaystyle{\mc P^{x}(x) = \prod_{s_{\ga}x \ed{} x} \ga}$, where the product is over the edges ending at $x$; and
\item[(2)] if $\mc P^{x}(y) \ne 0$, then $y \ge_{\h} x$.
\end{itemize}  
\end{definition}

These classes have been studied previously by Guillemin and Zara for a general construction of GKM rings \cite{GuZa-E}.  If for every $w \in W$ flow-up classes exist (which is not always true) the family forms a basis of $H^{*}_{T}(\h)$ as a free-$\R[\ga_{1},\cdots, \ga_{k}]$ module \cite{GuZa-E}.  Fortunately, for $H^{*}_{T}(\h)$ flow-up classes always exist. 

\begin{theorem}\label{Geo} Let $\h$ be a Hessenberg set, then the GKM ring $H^{*}_{T}(\h)$ has a basis of flow-up classes.
\end{theorem}

\begin{proof}  This follows because the GKM rings $H^{*}_{T}(\h)$ are the equivariant cohomology of the regular semisimple Hessenberg variety \cite{Te-HRY}, and \cite[Theorem 8]{dMPS-HV} proves for each $i$ that $\text{rank}_{\R[\Delta]} H_{T}^{i}(\h)$ satisfy \cite[Theorem 2.1]{GuZa-E}.   
\end{proof}

A drawback of this Theorem (besides its intentionally opaque nature) is that it only guarantees the existence of a flow-up basis.  We are still left with the problem of constructing the basis elements.  The construction of flow-up classes for GKM rings is important for several reasons.  First, an open problem of Schubert calculus is to determine the coefficients $c_{u v}^{w}$ defined in the expansion of the product of Schubert classes $\mc S^{u} \mc S^{v} = \sum c_{u v}^{w} \mc S^{w}$, so constructing generalized Schubert classes presents a new context to study this problem.  Second, flow-up classes form a basis of the representation of $W$ and without knowing a basis it will be essentially impossible to study the representation.

\begin{center}
\begin{figure}[ht]
\includegraphics[width = 145mm]{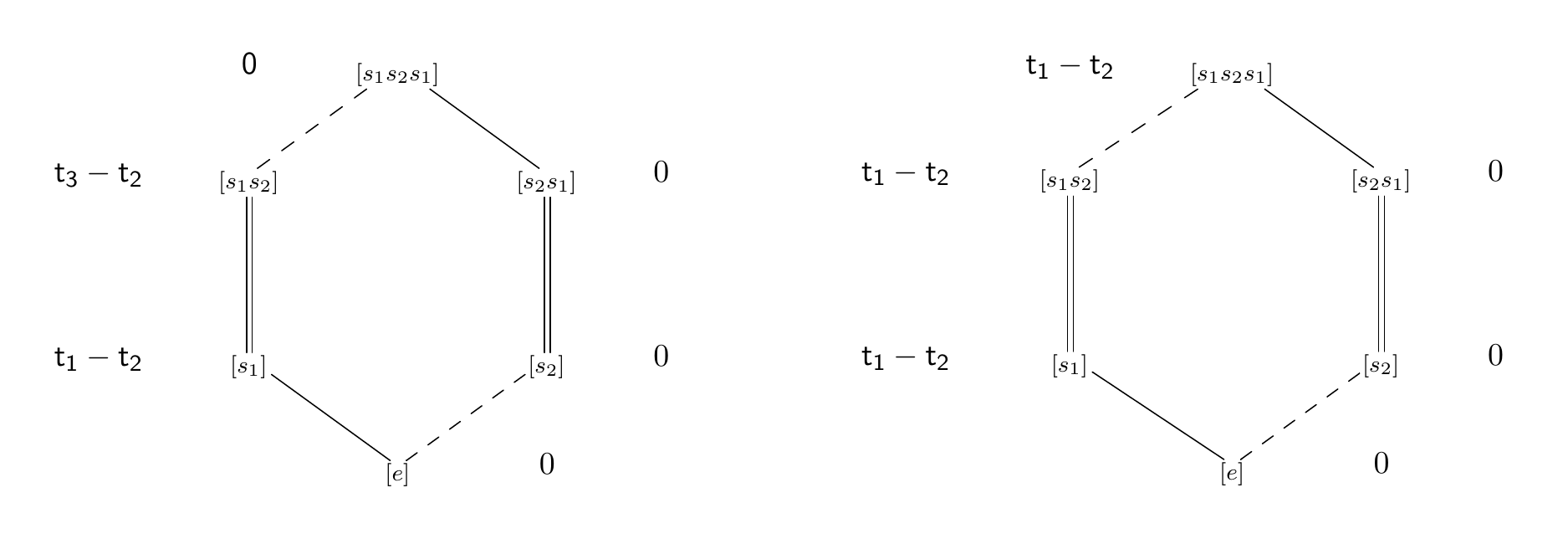}
\caption{Non-unique flow-classes}\label{unun}
\end{figure}
\end{center}

There do exist algorithms for the polynomials $\mc P^{x}(y)$ in general GKM rings \cite{GuZa-E, GoTo-GS}.  We adopt an alternative approach which emulates the construction of Schubert classes.  We use the representation of $W$ on $H^{*}_{T}(\h)$ (defined in Equation (\ref{dot})) to recursively build a new flow-up class.  This allows us to define a divided difference operator which as in the classical case recursive defines the flow-up class, \textit{i.e.} if $\mc P^{w}$ is know and $s_{i}w < w$, then $\partial_{i}^{\gamma}\mc P^{w} = \mc P^{s_{i}w}$.  A fundamental difficulty for us is that for a fixed $w \in W$ a flow-up class at $w$ is not unique (cf Figure \ref{unun}), a property enjoyed by Schubert classes.

\subsection{$\h$-inversions}
\label{h-inv}

The \emph{inversions of $w$} \textit{i.e.} $N_{w} : = \left\{\ga \in \Phi^{+} \mid w^{-1}\ga \in \Phi^{-} \right \}$ describe the edges ending at $w$ in the Bruhat graph.  This motivates the following 

\begin{definition}  Let $\h$ be a Hessenberg set.  For $w \in W$ the set \begin{math}N^{\h}_{w} := \left \{ \ga \in \Phi^{+} \mid w^{-1}\ga \in -\h \right \} \end{math}
is called the \textbf{$\h$-inversions} of $w$.
\end{definition}

The roots in $N^{\h}_{w}$ describe the edges ending at $w$ in $\Gamma_{\h}$, so knowing only $N_{w}^{\h}$ for all $w \in W$ alone determines the GKM ring.  Therefore, it is important to understand how $\h$-inversions change as $w \in W$ varies.

\begin{definition}  Let $w, v \in W$, we say $v$ is a \eb{cover} of $w$ if $w \ed{} v \in \Gamma_{W}$ and 
\begin{itemize}
\item[(1)] $\ell(v) = \ell(w) + 1$ and
\item[(2)] $v = s_{\ga}w$ 
\end{itemize}
\end{definition}

The following Proposition determines how the set $N_{w}$ and $N_{v}$ differ when $v$ is a cover of $w$ (cf. \cite{Ty-RS}).    

\begin{proposition}\label{inv}  Suppose $v$ is a cover of $w$, then 
\begin{displaymath}N_{v} = \{\ga\} \cup (s_{\ga}N_{w} \cap \Phi^{+}) \cup (N_{w} \cap s_{\ga}\Phi^{-}).\end{displaymath} 
\end{proposition}

This Proposition generalizes to $\h$-inversions. 

\begin{proposition}\label{h roots}  Suppose $v$ is a cover of $w$.  For $\gb \in N_{v}$ and

\begin{itemize}
\item[(1)]if $\gb \in s_{\ga}N_{w} \cap \Phi^{+}$ it follows $\gb \in N^{\h}_{v}$ if and only if $s_{\ga}\gb \in N^{\h}_{w}$ or
\item[(2)] if $\gb \in N_{w} \cap s_{\ga}\Phi^{-}$ it follows when $\gb \in N^{\h}_{v}$ then $\gb \in N^{\h}_{w}$.
\end{itemize} 
\end{proposition}

\begin{proof}  For Part $(1)$ if $s_{\ga}\gb \in N_{w}$ the equivalence follows because $v^{-1}\gb = w^{-1}s_{\ga}\gb$.  

For Part $(2)$ we show $v^{-1}\gb \prec w^{-1}\gb$ which by definition of $\h$ implies $w^{-1}\gb \in \h$ because $v^{-1}\gb \in \h$.  The hypothesis $\gb \in N_{w} \cap s_{\ga}\Phi^{-}$ implies $s_{\ga}\gb = \gb - c_{\ga\gb}\ga \in \Phi^{-}$, so the Cartan integer $c_{\ga \gb} > 0$.   Therefore, since $v = s_{\ga}w$ we have $w^{-1}\gb - v^{-1}\gb = c_{\ga \gb}w^{-1}\ga$.  Since $v$ is a cover of $w$ and $v = s_{\ga}w$ it follows $w^{-1}\ga \in \Phi^{+}$ which implies $v^{-1}\gb \prec w^{-1}\gb$.
\end{proof}

\begin{corollary}\label{simple}  Let $v$ be a cover of $w$.  If $\ga \in \Delta$ and $\ga \in N_{v}^{\h}$, then $N_{v}^{\h} = \{\ga\} \cup s_{\ga}N_{w}^{\h}$, otherwise if $\ga \not\in N_{v}^{\h}$, then $N_{v}^{\h} = s_{\ga}N_{w}^{\h}$.  
\end{corollary}

\begin{corollary}\label{v length}
Let $v$ be a cover of $w$, then $\ord{N_{v}} - \ord{\Phi^{-} \setminus \h} \le \ord{N_{v}^{\h}} \le \ord{N_{w}^{\h}} + 1$
\end{corollary}

The next Proposition determines the values of flow-up classes at the covers in the Bruhat order.  It is key to constructing a family of flow-up classes later. 

\begin{proposition}\label{v flow}  Let $\mc P$ be any flow-up class at $w$.  Suppose $v$ is a cover of $w$, then $\mc P(v)$ can be determined as follows
\begin{itemize}
\item[(1)] if $v^{-1}(\ga) \not\in \h$ then $\mc P(v) = 0$\; (\textit{i.e.} the edge $w \ed{} v \in \Gamma_{W}$ is deleted in $\Gamma_{\h}$); otherwise
\item[(2)] if $\ga \in \Delta \cap N^{\h}_{v}$ then $\mc P(v) = s_{\ga}\mc P(w)$;
\item[(3)] if $\ord{N_{v}^{\h}} = \ord{N_{w}^{\h}} + 1$ then 
\begin{displaymath}
\mc P(v) = \prod_{\gb \in N_{v}^{\h} \setminus \left\{\ga\right\}} \gb; \; \text{or}
\end{displaymath}
\item[(4)] if $\ord{N_{v}^{\h}} \le \ord{N_{w}^{\h}}$ then
\begin{displaymath}
\mc P(v) = f \prod_{\gb \in N_{v}^{\h} \setminus \left\{\ga\right\}} \gb
\end{displaymath}
for some $f \in \R[\Delta]$ of degree $\ord{N_{w}^{\h}} - \ord{N_{v}^{\h} \setminus \{\ga\}}$ with $f \equiv\prod_{\mu \in (N_{w}^{\h} \cap N_{v}) - N_{v}^{\h}}\mu \,\pmod{ \langle \ga \rangle}$.
\end{itemize}
\end{proposition}

\begin{proof}  Use Proposition \ref{h roots} and the GKM conditions to determine these values.
\end{proof}

\section{Highest root Hessenberg sets}
\label{root}

Suppose $\Phi$ is an irreducible root system, \textit{i.e.} $\Phi$ cannot be expressed as a disjoint union $\Phi = \Psi \cup \Psi'$ both of which are root systems.  For $\Phi$ irreducible there exists a unique \textbf{highest root} $\gamma \in \Phi^{+}$ such that $\ga \prec \gamma$ for all $\ga \in \Phi$ \cite[Section 2.9(3)]{Hu-RG}. If $\h_{\gamma} = \Phi^{+} \setminus \left\{ \gamma\right\}$, then $\h_{\gamma}$ is a Hessenberg set.

For $w \in W$ let $N^{\gamma}_{w} = N^{\h_{\gamma}}_{w}$ and $\ell_{\gamma}(w) = \ell_{\h_{\gamma}}(w)$.  We will be working with both the partial order defined by the flow-up $<_{\gamma}$ and the Bruhat order $<$.  Working with the highest root Hessenberg set simplifies much of the variation which occurs between $N_{w}$ and $N^{\gamma}_{w}.$  For example

\begin{lemma}\label{gamma len}  Suppose $v> w \in W.$   We have $\ell_{\gamma}(w) = \ell_{\gamma}(v)$ if and only if $v$ is a cover of $w$; $N_{w} = N_{w}^{\gamma}$; and there exists $\gb \in N_{v}$ such that $v^{-1}\gb = -\gamma$.
\end{lemma}

\begin{proof}  The converse follows by definition.  Therefore, suppose $\ell_{\gamma}(w) = \ell_{\gamma}(v)$.  Since $\ord \h = \ord {\Phi^{+}} - 1$ we have inequality $\ell(v) - 1 \le \ell_{\gamma}(v) = \ell_{\gamma}(w) < \ell(v)$, which implies $\ell_{\gamma}(v) = \ell(v) -1$.  Therefore, there exists a $\gb \in N_{v}$ such that $v^{-1}\gb = -\gamma.$  Further, the equality $\ell_{\gamma}(v) = \ell_{\gamma}(w)$ forces equality in $\ell_{\gamma}(v) = \ell(v) -1 \ge \ell(w) \ge \ell_{\gamma}(w)$. Hence, $\ell(v) = \ell(w) +1$, \textit{i.e.} $v$ is a cover of $w$ and $N_{w} = N_{w}^{\gamma}$.
\end{proof}

This Lemma with Proposition \ref{h roots} identifies an inversion $\gb \in N_{w}^{\gamma}\cap N_{v}$ such that $v^{-1}\gb = -\gamma$.  For a fixed $\gb$, the $v$ of Lemma \ref{gamma len} is unique.

\begin{corollary}\label{unique}  Suppose $v> w$  and $\gb \in N_{w}^{\gamma} \cap N_{v}$.  If $\ell_{\gamma}(w) = \ell_{\gamma}(v)$ and $v^{-1}\gb = -\gamma$, then $v$ is unique. 
\end{corollary} 

We are now ready to state the main Theorem of this paper.   

\begin{theorem}\label{BIG} These exist $\R[\Delta]$-module divided difference operators $\partial_{i}^{\gamma}: H^{*}_{T}(\h_{\gamma}) \longrightarrow H^{*}_{T}(\h_{\gamma})$ and a family of flow-up classes $\left \{\mc P^{w}\right \}_{w \in W}$ such that

\begin{displaymath}
\partial_{i}^{\gamma}\mc P^{w} = 
\begin{cases}
\mc P^{s_{i}w}	&\textrm{if}\; s_{i}w < w;\\
0			&\textrm{if}\; s_{i}w> w.
\end{cases}
\end{displaymath}
Further, if $w = s_{i_{1}} \cdots s_{i_{n}}$ is any reduced expression for $w \in W$, then the operator $\partial_{w} := \partial_{i_{1}} \cdots \partial_{i_{n}}$ is well-defined.  In other words, if $w = s_{j_{1}}\cdots s_{j_{n}}$ is another reduced expression for $w$, then 
\begin{displaymath}
\partial_{i_{1}} \cdots \partial_{i_{n}} = \partial_{j_{1}} \cdots \partial_{j_{n}}.
\end{displaymath}
\end{theorem} 

\subsection{Proof of Theorem \ref{BIG}}

In order to prove Theorem \ref{BIG} we give an explicit formula for the divided difference operator.  With this we work by induction on the length function $\ell(w)$ to define simultaneously the action of the simple reflection $s_{i} \cdot$ on the previously defined  flow-up classes AND define a new flow-up class satisfying Theorem \ref{BIG}.  

The base case of our induction is the longest element $w_{\circ} \in W$ (cf \cite[Theorem 1.8]{Hu-RG}) for which it is straightforward to define a flow-up class.  Since $N_{w_{\circ}} = \Phi^{+}$ it follows $N_{w_{\circ}}^{\h} = \h$, so $\mc P^{w_{\circ}}$ is the class whose value at $w_{\circ}$ is the product of the roots in $\h$ and $0$ otherwise.  Proceeding by induction, suppose for all $w \in W$ with $\ell(w) \ge k$ that flow-up classes satisfying Theorem \ref{BIG} have been defined.   

First a bit of notation, we say $s_{\ga}w \lessdot w$ if $s_\ga w < w$ in Bruhat order and the edge $s_{\ga}w \ed{} w$ has been deleted in $\Gamma_{\h}$, or in other words $w^{-1}\ga = -\gamma$.

\begin{definition}[Formula for $\partial_{i}^{\gamma}$]\label{DDH}  Let $w \in W$ with $\ell(w) = k$ and suppose  $\left \{\mc P^{u}\right \}_{\ell(u) \ge k}$ are flow-up classes in $\HG$. For each $s_{i} \in \Delta$ define the $i^{th}$ divided difference operator by

\begin{equation}\label{partial}
\partial_{i}^{\gamma}\mc P^{w} = 
\begin{cases}
s_{i}\cdot \mc P^{w}	&\text{if}\; s_{i}w \lessdot w;\\
\dfrac{\mc P^{w} - s_{i} \cdot \mc P^{w} + c_{\ga \ga_{i}} \left (\mc P^{v} -  \mc P^{s_{i}v}\right )}{\ga_{i}} &\text{if}\; s_{i}w < w;\\
0				&\text{if} \; s_{i}w > w,
\end{cases}
\end{equation}
where $c_{\ga \ga_{i}}$ is the Cartan integer of $s_{\ga}(\ga_{i})$, and when $v \in W$ exists it is the unique cover of $w$ such that $\ell_{\gamma}(w) = \ell_{\gamma}(v)$ and $v^{-1}\ga_{i} = -\gamma$. 
\end{definition}

\begin{example}  For the type $A_{2}$ root system the highest root set is $\h = \Delta$.  The family of flow-up classes constructed by Definition \ref{DDH} is described in Table \ref{eg}.  The reader is encouraged to replicate this data, for guidance $\Gamma_{\Delta}$ is given in Figure \ref{mgraphs}.  

\begin{table}[ht]
\label{eg}
$$\begin{array}{|c|c|c|c|c|c|c|}
\hline
\mc P^{v}(w) &w =  e & s_{1} & s_{2}& s_{1}s_{2} & s_{2}s_{1} & s_{1}s_{2}s_{1}\\
\hline
\mc P^{e} & 1 & 1 & 1 & 1 & 1 & 1\\
\hline
\mc P^{s_{1}}& 0 & \mathsf{t_{1}- t_{2}} & 0 & \mathsf{t_{3} - t_{2}} & 0 & 0\\
\hline 
\mc P^{s_{2}}& 0 & 0 &\mathsf{t_{2}- t_{3}} & 0 & \mathsf{t_{2} - t_{1}} & 0\\
\hline
\mc P^{s_{1}s_{2}}& 0 & 0 &0 & \mathsf{t_{1} - t_{3}} & 0 & \mathsf{t_{1}-t_{2}}\\
\hline
\mc P^{s_{2}s_{1}}& 0 & 0 &0 & 0 & \mathsf{t_{1} - t_{3}} & \mathsf{t_{2}-t_{3}}\\
\hline
\mc P^{s_{1}s_{2}s_{1}}& 0 & 0 &0 & 0 & 0 & (\mathsf{t_{1}-t_{2}})(\mathsf{t_{2}-t_{3}})\\
\hline
\end{array}$$
\caption{The family of flow-up classes for $\h = \Delta$ in type $A_{2}$.}
\end{table}
\end{example}

It is not  obvious that $\partial_{i}^{\gamma}$ is a $\R[\Delta]$-module homorphism, this is a consequence of the next theorem.  We prove this in \cite{Te-HRDD}, its proof requires a careful case-by-case analysis.   

\begin{theorem}\label{action}  Suppose Theorem \ref{BIG} has uniquely determined $\mc P^{w}$ for $w \in W$ with $\ell(w) \ge k$, then  

\begin{displaymath}s_{i}\cdot \mc P^{w} = 
\begin{cases}
\mc P 		&\text{if} \: s_{i}w \lessdot w \: \text{or} \: s_{i}w \gtrdot w\\
\mc P^{w}				&\text{if} \: s_{i}w > w\\
\mc P^{w} - \ga_{i}\mc P + c_{\ga \ga_{i}}(\mc P^{v} - \mc P^{s_{i}v})	
				&\text{if} \:  s_{i}w < w
\end{cases}
\end{displaymath}
where $\mc P$ is a flow-up class at $s_{i}w$  and when $v \in W$ exists it is the unique cover of $w$ such that $\ell_{\gamma}(w) = \ell_{\gamma}(v)$ and $v^{-1}\ga_{i} = -\gamma$.  
\end{theorem}

This provides the inductive step to Theorem \ref{BIG}.  The consequence is that the class $\mc P$ is a new flow-up class at $s_{i}w$ where $\ell(s_{i}w) = k-1$.  Repeating this process for all $w'$ with $\ell(w') = k-1$ proves the induction.  In fact this process uniquely defines a flow-up class at $s_{i}w$.  Before proving the uniqueness we show $\partial_{i}^{\gamma}$ is a module map.  Since, $\R[\Delta]$ is a UFD over $\R[\ga_{1}, \cdots, \widehat\ga_{i}, \cdots, \ga_{k}]$, where $\widehat\ga_{i}$ means $\ga_{i}$ is removed, dividing by $\ga_{i}$ is well-defined.  Therefore, in the third case of Theorem \ref{action} there exists a well-defined flow-up class such that
\begin{displaymath}\partial_{i}^{\gamma}\mc P^{w} := \dfrac{\mc P^{w} - s_{i}\cdot \mc P^{w} + c_{\ga \ga_{i}}\left (\mc P^{v} - \mc P^{s_{i}v}\right)}{\ga_{i}}= \mc P.\end{displaymath} 

This proves

\begin{corollary} The divided difference operator
 \begin{math}\partial_{i}^{\gamma}: \HG \longrightarrow \HG\end{math} 
 is a $\R[\Delta]$-module homorphism.
\end{corollary}

The next is a technical Lemma we need frequently  (cf. \cite[Lemma 5.11]{Hu-RG}).  It is important because it says that left multiplication by a simple transposition $s_{i}$ preserves the flow-up, \textit{i.e.} if $v$ is a cover of $w$, and $w \ed{} s_{i}w$ if and only if  $s_{i}v$ is a cover of $s_{i}w$.

\begin{lemma}[Diamond Lemma]\label{diamond}  Let $v$ be a cover of $w$.  Suppose $\ell(s_{i}w) = \ell(w) + 1 = \ell(v)$ and $s_{i}w \ne v$, then both $s_{i}v > v$ and $\ell(s_{i}v) = \ell(s_{i}w) + 1$.  Further, $w \ed{} v$ is in $\Gamma_\h$ if and only if $s_{i}w \ed{} s_{i}v$ is in $\Gamma_\h$.  
\end{lemma}

Next, we prove the flow-up class $\mc P$ defined in Theorem \ref{BIG} is uniquely determined.  This requires a new induction, which again our base case is $w_{\circ}$ which is uniquely defined.  Suppose by induction that if $\ell(v) > k$ that the flow-up classes are uniquely determined, and let $\ell(w) = k$.  This next Proposition determines the polynomials at all the covers of $w$ for the flow-up class $\mc P$ defined in Definition \ref{DDH}.  We include the proof as an example of how to prove these results.   

\begin{proposition}\label{p at v}  Suppose $\mc P$ is a flow-up classes at $w \in W$ defined by Definition \ref{DDH}, \textit{i.e.} $\mc P = \partial_{i}^{\gamma}\mc P^{s_{i}w}$ for $\ell(s_{i}w) = \ell(w)+1$.  Whenever $v$ covers $w$, then 
 \begin{equation}\label{v val}
 \mc P(v) = 
 \begin{cases}
 \displaystyle{s_{\ga}\mu \prod _{\gb \in N_{v}^{\gamma} \setminus \left\{\ga\right\}} \gb}  &\text{if} \; \ga \in N^{\gamma}_{v},\\
 0											&\textrm{if} \; \ga \not\in N^{\gamma}_{v}
 \end{cases}
\end{equation}
where $\mu \in (N_{w}^{\gamma} \cap N_{v}) \setminus N_{v}^{\gamma}$ or $\mu = 1$ otherwise. 
\end{proposition}

\begin{proof}  When $\mu$ exists it is the root associated to the edge $s_{\mu}v \ed{} v$ missing in $\Gamma_{\h_\gamma}$.  Define the polynomial $q = s_{\ga}\mu \prod _{\gb \in N_{v}^{\gamma} \setminus \left\{\ga\right\}} \gb$. 
When $\ga \not\in N_{v}^{\gamma}$ or $\ell_{\gamma}(v) = \ell_{\gamma}(w) + 1$ (when $\mu$ does not exist) this is proved in Proposition \ref{v flow}(1)-(3).  

Therefore, we may assume $\ell_{\gamma}(v) = \ell_{\gamma}(w)$, $\ga \in N_{v}^{\gamma}$, and $\mu \in (N_{w}^{\gamma} \cap N_{v}) \setminus N_{v}^{\gamma}$ exists.  We work by induction, for $w = w_{\circ}$ there is nothing to prove.  Suppose by induction for $w' \in W$ with $\ell(w') > k$ the result is true, and let $w \in W$ with $\ell(w) = k$.  In this case, there exists a simple reflection $s_{i}$ so that $\ell(s_{i}w) = \ell(w) + 1$.  By Lemma \ref{diamond}, $s_{i}v > v$, and $s_{i}v = s_{s_{i}\ga}w$, therefore $\mc P^{s_{i}w}(s_{i}v)$ satisfies the inductive hypothesis.   
 
If $s_{i}w \gtrdot w$, then by Equation (\ref{partial}) $\mc P := s_{i} \cdot \mc P^{s_{i}w}$.  It follows from Corollary \ref{simple} that $N^{\gamma}_{s_{i}v} = s_{i}N^{\gamma}_{v}$, so deduce $\mc P^{s_{i}w}(s_{i}v) = s_{i}q$. This shows $\mc P(v) = s_{i}\mc P^{s_{i}w}(s_{i}v) = q$ as desired. 

If $s_{i}w > w$, Equation (\ref{partial}) gives $\ga_{i}\mc P =  \mc P^{s_{i}w} - s_{i} \cdot \mc P^{s_{i}w} + c_{\gb \ga_{i}}(\mc P ^{v'} - \mc P^{s_{i}v'})$ where $v'$ may or may not exist.  Evaluating both sides of this expression at $v$ we claim
\begin{displaymath}\ga_{i}\mc P(v) = -s_{i}\mc P^{s_{i}w}(s_{i}v).\end{displaymath}
  
To prove this first note $\mc P^{s_{i}w}(v) = \mc P^{v'}(v) = 0$ since $v$ is not in the flow-up.  Next, when $\mc P^{s_{i}v'}(v) \ne 0$ since $\ell(s_{i}v') = \ell(v)$ it must be that $s_{i}v' = v$. This leads to a contradiction.  The hypothesis on $v'$ is that $s_{i}v' \lessdot v'$, but $s_{i}v' = v$ and $\ell_{\gamma}(v) = \ell(v) -1$.  This means at vertex $v$ in $\Gamma_{\h_{\gamma}}$ there are two edges deleted from the $\Gamma_{\h_{\gamma}}$, \textit{i.e.} $v^{-1}$ maps two roots to $-\gamma$, a contradiction since $v$ is invertible. 

Therefore, $\ga_{i}\mc P(v) = -s_{i}\mc P^{s_{i}w}(s_{i}v)$, and the inductive hypothesis shows $\mc P^{s_{i}w}(s_{i}v)$ is the product $s_{s_{i}\ga}s_{i}\mu = s_{i}s_{\ga}\mu$ times the product of the roots in $N_{s_{i}v}^{\gamma} = \{\ga_{i}\} \cup s_{i}N_{v}^{\gamma}$ except $s_{i}\ga$. Equivalently $\mc P(v)$ is the product of $s_{\ga}\mu$ and the roots in $N_{v}^{\gamma}$ except $\ga$, which is $q$ as desired.
\end{proof}

This will prove no matter how you arrive at $w$ the class $\mc P$ is uniquely determined.

\begin{corollary}\label{equal}  The flow-up classes defined by Definition \ref{DDH} are unique, \textit{i.e.} if $sv = w = tu$ where $s,t$ are simple reflections, then $\partial_{s}\mc P^{sv} = \mc P =  \partial_{t}\mc P^{tu}$.
\end{corollary}

\begin{proof} Let $\partial_{s}\mc P^{sv} = \mc P$ and $\partial_{t}\mc P^{su} = \mc P'$.  We want to show $\mc P = \mc P'$.  Since $\mc P'$ is non-zero only on $x >_{\gamma} w$ and homogeneous of degree $\ell_{\gamma}(w)$ we have a $\R[\Delta]$-linear combination
\begin{displaymath}\mc P' = \sum_{\tiny{\begin{array}{c} x >_{\gamma} w\\ \ell_{\gamma}(x) \le \ell_{\gamma}(w)\end{array}}} f_{x}\mc P^{x} + f_{w}\mc P.\end{displaymath}
Evaluating both sides of this expression at $w$ we have $\mc P'(w) = f_{w}\mc P(w)$, but $\mc P'(w) = \mc P(w)$ which determines that $f_{w} = 1$.  Next, evaluation at any $x >_{\gamma} w$ in the summation gives
\begin{displaymath}\mc P'(x) = f_{x}\mc P^{x}(x) + \mc P(x).\end{displaymath}
Since $\mc P^{x}(x) \ne 0$ and $\mc P'(x) = \mc P(x)$ by Proposition \ref{p at v} we conclude all the $f_{x} = 0$.  Therefore $\mc P' = \mc P$.
\end{proof} 
   
As a consequence we can define a unique class $\mc P^{s_{i}w} := \partial_{i}^{\gamma}\mc P^{w}$, and by induction this proves the first half of Theorem \ref{BIG}; that is there exists a family of flow-up classes $\{\mc P^{w}\}_{W}$ and divided difference operators $\partial_{i}^{\gamma}$.  Next, we prove the second half, that is if $w = s_{i_{1}}\cdots s_{i_{n}}$ is a reduced expression, then $\partial_{w} = \partial_{i_{1}}\cdots \partial_{i_{n}}$ is independent of the reduced expression.

\begin{theorem}  If $w \in W$ and $w = s_{i_{1}}\cdots s_{i_{n}}$ a reduced expression, then $\partial_{w} := \partial_{i_{1}}\cdots \partial_{i_{n}}$ is independent of the reduced expression, that is if $w = s_{j_{1}}\cdots s_{j_{n}}$, then 
\[\partial_{i_{1}}\cdots \partial_{i_{n}} = \partial_{j_{i}}\cdots \partial_{\j_{n}}.\]
\end{theorem}  

\begin{proof}[Sketch] Since any two expressions for $w \in W$ can obtained by a sequence of braid relations \cite[Theorem 1.9]{Hu-RG} it suffices to check if the Theorem is true for the braid relations.  Therefore, suppose that $v = stst\cdots = tsts\cdots$ and let $u$ and $u'$ be suffixes of $v$,  \textit{i.e.}  $s u = v = tu'$ such that $\ell(u) = \ell(v) - 1 = \ell(u')$.  Then, $\partial_{u}$ and $\partial_{u'}$ are well-defined since they have unique expressions in terms of the simple reflections.  To show $\partial_{v}$ is well-defined it suffices to show $\partial_{s}\partial_{u} = \partial_{t}\partial_{u'}$ by acting on the basis $\left \{ \mc P^{w}\right\}_{w \in W}$.

Now, we need only check the $x \in W$ such that $\ell(vx) = \ell(x) - \ell(v)$ or else by induction with Definiton \ref{DDH} 
\begin{displaymath}\partial_{s}\partial_{u}\mc P^{x} = 0 = \partial_{t}\partial_{u'}\mc P^{x}.\end{displaymath}  
In this case, we have $\ell(ux) = \ell(x) - \ell(u)$ and $\partial_{u}\mc P^{x} = \mc P^{ux}$ (respectively for $u'$).  The product $vx$ is well-defined, so conclude $sux = vx < ux$ if and only if $tu'x = vx < u'x$.  An application of Corollary \ref{equal} proves $\partial_{s}\partial_{u}\mc P^{x} = \partial_{s}\mc P^{ux}  = \mc P^{vx} = \partial_{t}\mc P^{u'x} = \partial_{t}\partial_{u'}\mc P^{x}$.     
\end{proof}

\subsection{Future work}

This work provides a model construction of divded difference operators and flow-up classes for all the GKM rings $H^{*}_{T}(\h)$.  A difficulty which needs to be overcome before we can obtain the equivalent of Theorem \ref{action} we need a better understanding of flow-up classes then Proposition \ref{v flow} provides.  Namely, here we take advantage that covers of $w$ essential determine $\mc P^{w}$.  In general, we will need to understand flow-up classes further up the flow of $w$ then just at the covers. 

An advantage of this approach is that it does determine the representation on $H^{*}_{T}(\h)$ when $\Phi$ is simply-laced, \text{i.e.} all the Cartan integers $c_{\ga \gb} = \pm 1$.  This next result will appear in \cite{Te-HRDD}.

\begin{theorem}\label{confirm} Suppose $\Phi$ is simply-laced.  Let $m_{V} = \frac{\ord{W}}{\ord{\Phi}}$ and $m_{\R} = \ord{W} - \ord{\Delta}m_{V}$, then as a $W$-module
\begin{displaymath}
 H^{*}_{T}(\h) = (V^{\oplus m_{V}} \bigoplus \R^{\oplus m_{\R}}) \; \bigotimes_{\R}  \;\R[{\Delta}],
\end{displaymath}
where $V$ is the reflection representation (cf. Section \ref{hgraph}) , $\R$ is the trivial representation and $\R[\Delta]$ is the polynomial representation of $W$. 
\end{theorem}

In the case where $\Phi$ is the type $A$ root system we have

\begin{theorem} If $\Phi$ is the type $A_{n-1}$ root system, then as a $\mf S_{n}$-module

\begin{displaymath}
 H^{*}_{T}(\h) = (V^{\oplus (n-2)!} \bigoplus \R^{\oplus (n-1)!(n-1)})  \; \bigotimes_{\R} \;\R[{\Delta}].
\end{displaymath}
Furthermore, this proves the Shareshian-Wachs conjecture \cite[Conjecture 5.3]{ShWa-CQH}.
\end{theorem}

\bibliographystyle{plain}   
\bibliography{bib}

\end{document}